\documentclass[12pt]{amsart}
\usepackage{pxfonts}
\usepackage{color}

\theoremstyle{plain}
\numberwithin{equation}{section}
\newtheorem{theo}{Theorem}[section]
\newtheorem{lemm}[theo]{Lemma}
\newtheorem{coro}[theo]{Corollary}
\newtheorem{prob}[theo]{Problem}
\newtheorem{prop}[theo]{Proposition}
\theoremstyle{remark}
\newtheorem{rema}[theo]{Remark}
\def\C{\mathbb C}
\def\Z{\mathbb Z}
\def\R{\mathbb R}
\DeclareMathOperator{\Todd}{\mathrm{Todd}}
\DeclareMathOperator{\link}{link}

\DeclareMathOperator{\sta}{star}

\DeclareMathOperator{\pos}{pos}
\DeclareMathOperator{\Lie}{Lie}
\DeclareMathOperator{\e}{\mathbf e}

\def\Tn{T^n}
\def\T2{T^2}

\begin{document}
	\title{Todd genera of complex torus manifolds}
	\author[H.~Ishida]{Hiroaki Ishida}
	\address{Osaka City University Advanced Mathematical Institute, Sumiyoshi-ku, Osaka 558-8585, Japan.}
	\email{ishida@sci.osaka-cu.ac.jp}
	\author[M.~Masuda]{Mikiya Masuda}
	\address{Department of Mathematics, Osaka City University, Sumiyoshi-ku, Osaka 558-8585, Japan.}
	\email{masuda@sci.osaka-cu.ac.jp}

	\date{\today}
	\thanks{The first author was supported by JSPS Research Fellowships for Young Scientists}
	\thanks{The second author was partially supported by Grant-in-Aid for Scientific Research 22540094}
	\keywords{Torus manifold, complex manifold, toric manifold}
\subjclass[2010]{Primary 57R91, Secondary 32M05, 57S25}

	\begin{abstract}
In this paper, we prove that the Todd genus of a compact complex manifold $X$ of complex dimension $n$ with vanishing odd degree cohomology is one if the automorphism group of $X$ contains a compact $n$-dimensional torus $\Tn$ as a subgroup.  This implies that if a quasitoric manifold admits an invariant complex structure, then it is equivariantly homeomorphic to a compact smooth toric variety, which gives a negative answer to a problem posed by Buchstaber-Panov.   		
	\end{abstract}
	\maketitle

	\section{Introduction}
	A \emph{torus manifold} is a connected closed oriented smooth manifold of even dimension, say $2n$, endowed with an effective action of an $n$-dimensional torus $\Tn$ having a fixed point. A typical example of a torus manifold is a \emph{compact smooth toric variety} which we call a \emph{toric manifold} in this paper. Every toric manifold is a complex manifold. However, a torus manifold does not necessarily admit a complex (even an almost complex) structure. For example, the $4$-dimensional sphere $S^4$ with a natural $\T2$-action is a torus manifold but admits no almost complex structure.

	On the other hand, there are infinitely many nontoric torus manifolds of dimension $2n$ which admit $\Tn$-invariant almost complex structures when $n\ge 2$.  For instance,  for any positive integer $k$, there exists a torus manifold of dimension $4$ with an invariant almost complex structure whose Todd genus is equal to $k$ (\cite[Theorem 5.1]{Mas99}) while the Todd genus of a toric manifold is always one.  One can produce higher dimensional examples by taking products of those $4$-dimensional examples with toric manifolds.  The cohomology rings of the torus manifolds in these examples are generated by its degree-two part like toric manifolds.    

	In this paper, we consider a torus manifold with a $\Tn$-invariant (genuine) complex structure. We will call such a torus manifold a \emph{complex torus manifold}.  The following is our main theorem.    
	\begin{theo} \label{theo:maintheo}
		If a complex torus manifold has vanishing odd degree cohomology, then its Todd genus is equal to one.
	\end{theo}
	\begin{rema}\label{rema:Lefschetz}
		If a closed smooth manifold $M$ has vanishing odd degree cohomology, then any smooth $T^n$-action on $M$ has a fixed point (see \cite[Corollary 10.11 in p.164]{Bre72}).  
		In particular, a connected closed oriented smooth manifold $M$ of dimension $2n$ with an effective $T^n$-action is a torus manifold if $M$ has vanishing odd degree cohomology.  This implies that Theorem~\ref{theo:maintheo} is equivalent to the statement in the abstract.  
		
	\end{rema}
	Other important examples of torus manifolds are \footnote{Davis-Januszkiewicz \cite{Da-Ja91} uses the terminology \emph{toric manifold} but it was already used in algebraic geometry as the meaning of (compact) smooth toric variety, so Buchstaber-Panov \cite{Bu-Pa02} started using the word \emph{quasitoric manifold}.}\emph{quasitoric manifolds} introduced by M.~W.~Davis and T.~Januskiewicz (\cite{Da-Ja91}). A quasitoric manifold of dimension $2n$ is a closed smooth manifold with a locally standard $\Tn$-action, whose orbit space is an $n$-dimensional simple polytope.  It is unknown whether any toric manifold is a quasitoric manifold.  However, if a toric manifold is projective, then it is a quasitoric manifold because a projective toric manifold with the restricted compact torus action admits a moment map which identifies the orbit space with a simple polytope.  

A.~Kustarev (\cite[Theorem 1]{Kus09}) gives a criterion of when a quasitoric manifold admits an invariant almost complex structure.  It also follows from his criterion that there are many nontoric quasitoric manifolds which have invariant almost complex structures. 
However, it has been unknown whether there is a quasitoric manifold which admits an invariant complex structure, and V.~M.~Buchstaber and T.~E.~Panov posed the following problem  (\cite[Problem 5.23]{Bu-Pa02}), which motivated the study in this paper. 
	\begin{prob}[Buchstaber-Panov]\label{prob:BuPa}
		Find an example of nontoric quasitoric manifold that admits an invariant complex structure.
	\end{prob}

As a consequence of Theorem~\ref{theo:maintheo}, we obtain the following which gives a negative answer to Problem \ref{prob:BuPa}.
	\begin{theo}\label{theo:quasitoric}
		If a quasitoric manifold admits an invariant complex structure, then it is equivariantly homeomorphic to a toric manifold.
	\end{theo}

	This paper is organized as follows. In Section \ref{sec:dim2}, we study simply connected compact complex surfaces with torus actions.  In Section \ref{sec:multi-fan}, we review the notion of multi-fan and recall a result on Todd genus. In Section \ref{sec:map}, we define a map associated with the multi-fan of a complex torus manifold $X$ and give a criterion of when the Todd genus of $X$ is one in terms of the map. Theorems \ref{theo:maintheo} and \ref{theo:quasitoric} are proved in Sections \ref{sec:oddvanish}  and  \ref{sec:quasitoric} respectively.  Throughout this paper, all cohomology rings and homology groups are taken with   $\Z$-coefficients.
	
While preparing this paper, the first author and Yael Karshon proved that a complex torus manifold is equivariantly biholomorphic to a toric manifold (\cite{Ih-Ka12}).  Although Theorem~\ref{theo:maintheo} is contained in the result, the argument in this paper is completely different from that in \cite{Ih-Ka12} and we believe that this paper is worth publishing.

\section{Simply connected complex surfaces with torus actions}\label{sec:dim2}

We first recall two results on simply connected 4-manifolds.  

	\begin{theo}[\cite{Or-Ra70}]\label{theo:Or-Ra}
	If a simply connected closed smooth manifold of dimension $4$ admits an effective smooth action of $\T2$, then it is diffeomorphic to a connected sum of copies of  $\C P^2$, $\overline{\C P^2}$ ($\C P^2$ with reversed orientation) and $S^2\times S^2$. 
	\end{theo}

		\begin{theo}[\cite{Don90}]\label{theo:Donaldson}
			If a simply connected projective complex surface is decomposed into $Y_1\# Y_2$ as oriented smooth manifolds, then either $Y_1$ or $Y_2$ has a negative definite cup-product form. 
		\end{theo}

	Let $X$ be a simply connected compact complex surface whose automorphism group contains $\T2$ as a subgroup. 
	By Theorem \ref{theo:Or-Ra}	
\begin{equation} \label{eq:2.0}
		X \cong k\C P^2 \# \ell\overline{\C P^2} \# m(S^2\times S^2), \quad k,\ell , m \geq 0
	\end{equation}
	as oriented smooth manifolds. Therefore, the Euler characteristic $\chi (X)$ and the signature $\sigma (X)$ of $X$ are respectively given by 
	\begin{equation*}
		\chi (X)= k+\ell +2m+2 \text{\quad and \quad} \sigma(X) =k-\ell
	\end{equation*}
	and hence the Todd genus $\Todd(X)$ of $X$ is given by 
	\begin{equation} \label{eq:2.1}
		\Todd(X) =\frac{1}{4}(\chi (X) +\sigma (X))= \frac{1}{2}(k +m +1).
	\end{equation}
	
	The following proposition is a key step toward Theorem~\ref{theo:maintheo}.  
	
	\begin{prop}\label{prop:todd}
	Let $X$ be as above.  Then $\Todd(X)=1$. 
	\end{prop}
	\begin{proof}
		Since $X$ is simply connected, the first betti number of $X$ is $0$, in particular, even. Thus, $X$ is a deformation of an algebraic surface (\cite[Theorem 25]{Kod64}). 
		Since any algebraic surface is projective (see \cite[Chapter IV, Corollary 5.6]{BPV84}), we can apply Theorem \ref{theo:Donaldson} to our $X$.  

Unless $(k,m)=(1,0)$ or $(0,1)$, it follows from \eqref{eq:2.0} that $X$ can be decomposed into $Y_1 \# Y_2$ as oriented smooth manifolds, where 
		\begin{alignat*}{2}
			Y_1 &= \C P^2, \quad Y_2 = (k-1)\C P^2 \# \ell \overline{\C P^2} \# m(S^2\times S^2) \quad&\text{if $k\ge 2$},\\
			Y_1 &= S^2\times S^2,\quad Y_2=k\C P^2 \# \ell \overline{\C P^2} \# (m-1)(S^2\times S^2) \quad&\text{if $m\ge 2$},\\
			Y_1 &= \C P^2 \# \ell \overline{\C P^2},\quad Y_2 =S^2 \times S^2 \quad&\text{if $(k,m)=(1,1)$}.
		\end{alignat*}
In any case, neither of $Y_1$ and $Y_2$ has a negative cup-product form and this contradicts Theorem \ref{theo:Donaldson}.	Therefore, $(k,m)=(1,0)$ or $(0,1)$ and hence $\Todd(X)=1$ by \eqref{eq:2.1}.  
	\end{proof}

\section{Torus manifolds and multi-fans}\label{sec:multi-fan}
	In this section, we review the notion of multi-fans introduced in \cite{Ha-Ma03} and \cite{Mas99} and recall a result on Todd genus.  

A \emph{torus manifold} $X$ of dimension $2n$ is a connected closed oriented manifold endowed with an effective action of $\Tn$ having a fixed point. In this paper, we are concerned with the case when $X$ has a complex structure invariant under the action.  We will call such a torus manifold a \emph{complex torus manifold}.
	
	Throughout this section, $X$ will denote a complex torus manifold of complex dimension $n$ unless otherwise stated.  We define a combinatorial object $\Delta _X:= (\Sigma _X, C_X, w_X)$ called the \emph{multi-fan} of $X$. A \emph{characteristic submanifold} of $X$ is a connected complex codimension $1$ holomorphic submanifold of $X$ fixed pointwise by a circle subgroup of $\Tn$. Characteristic submanifolds are $\Tn$-invariant and intersect transversally.  Since $X$ is compact, there are only finitely many characteristic submanifolds, denoted $X_1,\dots ,X_m$. We set  
	\begin{equation*}
		\Sigma _X := \left\{ I \in \{1,2,\dots,m\} \mid X_I:=\bigcap _{i \in I}X_i \neq \emptyset \right\},
	\end{equation*}
which is an abstract simplicial complex of dimension $n-1$. 
	
	Let $S^1$ be the unit circle group of complex numbers and $T_i$ the circle subgroup of $\Tn$ which fixes $X_i$ pointwise. We take the isomorphism $\lambda _i : S^1 \to T_i \subset \Tn$ such that 
	\begin{equation} \label{eq:3.1}
		\lambda_i(g)_*(\xi )=g\xi \quad\text{for $^\forall g \in S^1 \text{ and } ^\forall \xi \in TX|_{X_i}/TX_i$}
	\end{equation} 
	where $\lambda _i(g)_*$ denotes the differential of $\lambda _i(g)$ and the right hand side of \eqref{eq:3.1} above is the scalar multiplication with the complex number $g$ on the normal bundle $TX|_{X_i}/TX_i$ of $X_i$.   We regard $\lambda_i$ as an element of the Lie algebra $\Lie\Tn$ of $\Tn$ through the differential and assign a cone 
	\begin{equation} \label{eq:3.2}
		C_X(I) := \pos (\lambda _i \mid i \in I) \subset \Lie\Tn
	\end{equation}
	to each simplex $I \in \Sigma _X$, where $\pos (A)$ denotes the positive hull spanned by elements in the set $A$.  This defines a map $C_X$ from $\Sigma _X$ to the set of cones in $\Lie\Tn$.
	
  We denote the set of $(n-1)$-dimensional simplices in $\Sigma _X$ by $\Sigma _X^{(n)}$.  For $I\in \Sigma^{(n)}$, $X_I$ is a subset of the $\Tn$-fixed point set of $X$. The \emph{weight function} $w_X\colon \Sigma _X^{(n)} \to \Z _{>0}$ is given by
	\begin{equation*}
		w_X(I):= \# X_I
	\end{equation*}
	where $\# A$ denotes the cardinality of the finite set $A$.

	The triple $\Delta _X := (\Sigma _X, C_X, w_X)$ is called the multi-fan of $X$. The Todd genus $\Todd(X)$ of $X$ can be read from the multi-fan $\Delta _X$ as follows. 
	\begin{theo}[\cite{Mas99}]\label{theo:todd}
Let $v$ be an arbitrary vector in $\Lie\Tn$ which is not contained in $C_X(J)$ for any $J\in \Sigma_X\backslash\Sigma_X^{(n)}$.  Then 
		\begin{equation*}
			\Todd(X)=\sum w_X(I)
		\end{equation*}
		where the summation runs over all $I \in \Sigma _X^{(n)}$ such that $C_X(I)$ contains $v$.  
	\end{theo}
	
The following corollary follows immediately from Theorem~\ref{theo:todd}.  

	\begin{coro}\label{coro:fan}
		$\Todd(X) =1$ if and only if the pair $(\Sigma _X,C_X)$ forms an ordinary complete nonsingular fan and $w_X(I)=1$ for every $I\in\Sigma^{(n)}_X$.
	\end{coro}
	
Suppose $X_J$ is connected for every $J \in \Sigma _X$.  Then $X_J$ is a complex codimension $\# J$ holomorphic submanifold of $X$ having a $\Tn$-fixed point. Moreover, the induced action of the quotient torus $\Tn/T_J$ on $X_J$ is effective and preserves the complex structure of $X_J$, where $T_J$ is the $\# J$-dimensional subtorus of $\Tn$ generated by $T_j$ for $j\in J$.  Therefore, $X_J$ is a complex torus manifold of complex dimension $n-\# J$ with the effective action of  the quotient torus $\Tn/T_J$. 
	
In this case, the multi-fan $\Delta _{X_J}=(\Sigma _{X_J},C_{X_J},w_{X_J})$ of $X_J$ for $J\in \Sigma$ can be obtained from the multi-fan $\Delta _X$ of $X$ as discussed in \cite{Ha-Ma03}, which we shall review.  We note that $X_J\cap X_i$ is non-empty if and only if $J\cup \{ i\}$ is a simplex in $\Sigma _X$ and each characteristic submanifold of $X_J$ can be written as the non-empty intersection $X_J\cap X_i$. Hence, the simplicial complex $\Sigma _{X_J}$ coincides with the link $\link (J;\Sigma _X)$ of $J$ in $\Sigma _X$ and 
	\begin{equation*}
		C_{X_J}(I)=\pos (\overline{\lambda _i}\mid i \in I)
	\end{equation*}
	for $I \in \link (J; \Sigma _X)$, where $\overline{\lambda _i}$ denotes the image of $\lambda _i$ by the quotient map $\Lie\Tn \to \Lie\Tn/T_J$. The weight function $w_{X_J}$ is the constant function $1$.

	\section{Maps associated with multi-fans}\label{sec:map}
Let $X$ be a complex torus manifold of complex dimension $n$ and $\Delta _X=(\Sigma _X, C_X,w_X)$ the multi-fan of $X$. Throughout this section, we assume that $X_J$ is connected for every $J\in \Sigma_X$. We will define a continuous map $f_X$ from the geometric realization $|\Sigma_X|$ of $\Sigma_X$ to the unit sphere $S^{n-1}$ of the vector space $\Lie\Tn$ in which the cones $C_X(I)$ for $I\in \Sigma_X$ sit, and give a criterion of when the Todd genus of $X$ is equal to $1$ in terms of the map $f_X$.   

	We set 
	\begin{equation*}
		\sigma _I := \left\{ \sum _{i \in I}a_i\e _i \mid \sum _{i\in I}a_i =1, a_i \geq 0 \right\} \subset \R ^m \quad\text{for $I \in \Sigma _X$}, 
	\end{equation*}
where $\e _i$ is the $i$-th vector in the standard basis of $\R ^m$. The geometric realization $|\Sigma_X|$ of $\Sigma _X$ is given by  
	\begin{equation*}
		|\Sigma _X| = \bigcup _{I \in \Sigma _X}\sigma _I.
	\end{equation*}	
Recall that the homomorphisms $\lambda _i\colon S^1\to \Tn$ for $i=1,\dots,m$ defined in Section~\ref{sec:multi-fan} are regarded as elements in $\Lie\Tn$ through the differential.  We take an inner product on $\Lie \Tn$ and denote the length of an element $v\in \Lie\Tn$ by $|v|$.  We define a map $f_X\colon |\Sigma _X|\to S^{n-1}$, where $S^{n-1}$ is the unit sphere of $\Lie\Tn$, by  
	\begin{equation} \label{eq:4.1}
		f_X|_{\sigma _I}\left( \sum_{i\in I}a_i\e _i\right) =\frac{\sum _{i\in I}a_i\lambda _i}{|\sum _{i\in I}a_i\lambda _i|}.
	\end{equation}
Clearly, $f_X$ is a closed continuous map. 
	\begin{lemm}\label{lemm:homeo}
		The map $f_X \colon |\Sigma _X| \to S^{n-1} $ is a homeomorphism if and only if $\Todd(X)=1$.
	\end{lemm}
	\begin{proof}
We note that $X_I$ is one point for any $I\in \Sigma_X^{(n)}$ because $X_I$ is connected by assumption and of codimension $n$ in $X$.  Therefore, $w_X(I)=1$ for any $I\in \Sigma_X^{(n)}$ and this together with Theorem~\ref{theo:todd} tells us that the Todd genus $\Todd(X)$ coincides with the number of cones $C_X(I)$ containing the vector $v \in \Lie\Tn$ in Theorem~\ref{theo:todd}. 

The above observation implies that the cones $C_X(I)$ for $I\in \Sigma _X$ do not overlap and form an ordinary complete fan in $\Lie\Tn$ if and only if $\Todd(X)=1$ and this is equivalent to the map $f_X$ being a homeomorphism, proving the lemma. 
	\end{proof}
	For each characteristic submanifold $X_i$, we can also define a map $f_{X_i} : |\Sigma _{X_i}| \to S^{n-2}$, where $S^{n-2}$ is the unit sphere in $\Lie\Tn/T_i \cong (\Lie T_i)^\perp$, where $(\Lie T_i)^\perp$ denotes the orthogonal complement of a vector subspace $\Lie T_i$ in $\Lie\Tn$. 
	\begin{lemm}
		If $f_{X_i} : |\Sigma _{X_i}| \to S^{n-2}$ is a homeomorphism, then $f_X|_{\sta (\{i\};\Sigma _X)} : {\sta (\{i\};\Sigma _X)} \to S^{n-1}$ is a  homeomorphism onto its image, where $\sta (\{i\};\Sigma _X)$ denotes the open star of $\{ i\}$ in $\Sigma _X$.
	\end{lemm}
	\begin{proof}
		It suffices to show the injectivity of $f_X|_{\sta (\{i\};\Sigma _X)}$ because $f_X$ is closed and continuous. 
		Let $p_i : \Lie\Tn \to (\Lie T_i)^\perp$ be the orthogonal projection. Through $p_i$, we identify $\Lie\Tn/T_i$ with $(\Lie T_i)^\perp$. Recall that $\Sigma _{X_i} = \link (\{ i\};\Sigma _X)$. For each vertex $j$ of $\link (\{ i\};\Sigma _X)$, we express $\lambda _j$ as
		\begin{equation*}
			\lambda _j = p_i(\lambda _j)+c_{i,j}\lambda _i, \quad c_{i,j} \in \R.
		\end{equation*}
		By the definitions of $\link (\{ i\};\Sigma _X)$ and $\sta (\{i\} ;\Sigma _X)$, we can express an element $x \in \sta (\{ i\} ;\Sigma _X)$ as
		\begin{equation} \label{eq:4.2}
			x = (1-t)\e _i+ty, \quad\text{with $y \in |\link (\{i\};\Sigma _X)|,\ 0\leq t <1$}.
		\end{equation}
		Suppose $y \in \sigma _J \subset |\link (\{i\};\Sigma _X)|$ and write
		\begin{equation*}
			y = \sum _{j \in J}a_j\e _j, \sum _{j \in J}a_j =1, a_j\geq 0.
		\end{equation*}
		Then, it follows from \eqref{eq:4.1} that  
\begin{equation} \label{eq:4}
		\begin{split}
			f_X(x) &=\frac{(1-t)\lambda _i+t\sum _{j \in J}a_j\lambda _j}{|(1-t)\lambda _i+t\sum _{j \in J}a_j\lambda _j|}\\
			&= \frac{(1-t)\lambda _i+t\sum _{j \in J}a_j(p_i(\lambda_j)+c_{i,j}\lambda _i)}{|(1-t)\lambda _i+t\sum _{j \in J}a_j(p_i(\lambda_j)+c_{i,j}\lambda _i)|}\\
			&= g(t,y)f_{X_i}(y)+h(t,y)\frac{\lambda _i}{|\lambda _i|},
		\end{split}
\end{equation}
		where 
		\begin{equation} \label{eq:4.g}
			g(t,y):= \frac{t|\sum _{j \in J}a_jp_i(\lambda_j)|}{|(1-t)\lambda _i+t\sum _{j \in J}a_j(p_i(\lambda_j)+c_{i,j}\lambda _i)|}
		\end{equation}
		and 
		\begin{equation*}
			h(t,y):= \frac{(1-t+t\sum_{j\in J}a_jc_{i,j})|\lambda _i|}{|(1-t)\lambda _i+t\sum _{j \in J}a_j(p_i(\lambda_j)+c_{i,j}\lambda _i)|}.
		\end{equation*}
		Since $|f_X(x)| = |f_{X_i}(y)|=1$ and $f_{X_i}(y)$ is perpendicular to $\lambda_i/|\lambda_i|$, it follows from \eqref{eq:4} that  
\begin{equation} \label{eq:4.2-1}
g^2(t,y)+h^2(t,y)=1.
\end{equation}
		
Take another point $x'\in\sta(\{i\},\Sigma_X)$ and write 
\[
			x' = (1-t)\e_i+t'y', \quad\text{with $y' \in |\link (\{i\};\Sigma _X)|,\ 0\leq t' <1$}
\]
similarly to \eqref{eq:4.2}.  Since 
\[
f_X(x')=g(t',y')f_{X_i}(y')+h(t',y')\frac{\lambda_i}{|\lambda_i|},
\]
we have $f_X(x)=f_X(x')$ if and only if 
\begin{equation} \label{eq:4.3}
g(t,y)f_{X_i}(y)=g(t',y')f_{X_i}(y')\quad\text{and}\quad h(t,y)=h(t',y').
\end{equation}
Both $g(t,y)$ and $g(t',y')$ are non-negative by \eqref{eq:4.g}, so it follows from  \eqref{eq:4.2-1} and \eqref{eq:4.3} that 
if $f_X(x)=f_X(x')$, then 
\begin{equation} \label{eq:4.4}
g(t,y)=g(t',y')\quad\text{and}\quad f_{X_i}(y)=f_{X_i}(y').
\end{equation}  
The latter identity in \eqref{eq:4.4} above implies $y=y'$ since $f_{X_i}$ is a homeomorphism by assumption.  Therefore it follows from \eqref{eq:4.3} and \eqref{eq:4.4} that   
\[
g(t,y)=g(t',y)\quad\text{and}\quad h(t,y)=h(t',y).
\]
This together with \eqref{eq:4} shows that 
\[
(1-t)\lambda _i+t\sum _{j \in J}a_j\lambda _j=(1-t')\lambda _i+t'\sum _{j \in J}a_j\lambda _j.
\]
Here $\lambda_i$ and $\sum_{j\in J}a_j\lambda_j$ are linearly independent, so we conclude $t=t'$.  It follows that $f_X|_{\sta (\{i\};\Sigma _X)}$ is injective, which implies the lemma. 
	\end{proof}
	We have the following corollary.
	\begin{coro}\label{coro:covering}
		If $f_{X_i} : |\Sigma _{X_i}| \to S^{n-2}$ is a homeomorphism for all $i$, then $f_X : |\Sigma _X| \to S^{n-1}$ is a covering map, and hence if $|\Sigma _X|$ is connected and $n-1 \geq 2$ in addition, then $f_X$ is a homeomorphism.
	\end{coro}

	\section{Torus manifolds with vanishing odd degree cohomology}\label{sec:oddvanish}
	In this section, we prove Theorem~\ref{theo:maintheo} in the Introduction. 

The $\Tn$-action on a torus manifold $X$ of dimension $2n$ is said to be \emph{locally standard} if the $\Tn$-action on $X$ locally looks like a faithful representation of $\Tn$, to be more precise, any point of $X$ has an invariant open neighborhood equivariantly diffeomorphic to an invariant open set of a faithful representation space of $\Tn$. The orbit space $X/\Tn$ is a manifold with corners if the $\Tn$-action on $X$ is locally standard. A manifold with corners $Q$ is called \emph{face-acyclic} if every face of $Q$ (even $Q$ itself) is acyclic.  A face-acyclic manifold with corners is called a \emph{homology polytope} if any intersection of facets of $Q$ is connected unless empty.   
The combinatorial structure of $X/\Tn$ and the topology of $X$ are deeply related as is shown in the following theorem. 

	\begin{theo}[\cite{Ma-Pa06}]\label{theo:Ma-Pa2}
	Let $X$ be a torus manifold of dimension $2n$.  
\begin{enumerate}
\item $H^{odd}(X)=0$ if and only if the $\Tn$-action on $X$ is locally standard and $X/\Tn$ is face-acyclic.
\item $H^*(X)$ is generated by its degree-two part as a ring if and only if the $\Tn$-action on $X$ is locally standard and $X/\Tn$ is a homology polytope. 
\end{enumerate}
	\end{theo}

Suppose that $X$ is a torus manifold of dimension $2n$ with vanishing odd degree cohomology.  Then $X/\Tn=Q$ is a manifold with corners and face-acyclic.  Let $\pi \colon X\to X/\Tn=Q$ be the quotient map and let $Q_1,\dots,Q_m$ be the facets of $Q$.  Then $\pi^{-1}(Q_1),\dots,\pi^{-1}(Q_m)$ are the characteristic submanifolds of $X$, denoted $X_1,\dots,X_m$ before.  If $Q$ is a homology polytope, i.e. any intersection of facets of $Q$ is connected unless empty (this is equivalent to any intersection of characteristic submanifolds of $X$ being connected unless empty), then the geometric realization $|\Sigma_X|$ of the simplicial complex $\Sigma_X$ is a homology sphere of dimension $n-1$ (see \cite[Lemma 8.2]{Ma-Pa06}), in particular, connected when $n\ge 2$.  Unless $Q$ is a homology polytope, intersections of facets are not necessarily connected.  However, we can change $Q$ into a homology polytope by cutting $Q$ along faces of $Q$.  This operation corresponds to blowing-up along connected components of intersections of characteristic submanifolds of $X$ equivariantly.  We refer the reader to \cite{Ma-Pa06} for the details.      

The results in Section~\ref{sec:dim2} required the simply connectedness of a complex surface.  Here is a criterion of the simply connectedness of a torus manifold in terms of its orbit space.   

\begin{lemm} \label{lemm:5.1}
Suppose that the $\Tn$-action on a torus manifold $X$ is locally standard.  Then 
$X$ is simply connected if and only if the orbit space $X/\Tn$ is simply connected.
\end{lemm}

\begin{proof}
Since the group $\Tn$ is connected, the \lq\lq only if" part in the lemma follows from \cite[Corollary 6.3 in p.91]{Bre72}.    

We shall prove the \lq\lq if" part.  Suppose that $X/\Tn$ is simply connected.  Since each characteristic submanifold $X_i$ of $X$ is of real codimension two, the homomorphism 
\begin{equation} \label{eq:5.1}
\pi_1(X\backslash \cup_i X_i)\to \pi_1(X)
\end{equation}
induced by the inclusion map from $X\backslash \cup_i X_i$ to $X$ is surjective.  Here, the $\Tn$-action on $X\backslash \cup_i X_i$ is free since the $\Tn$-action on $X$ is locally standard, so that the quotient map from $X\backslash\cup_i X_i$ to $(X\backslash\cup_i X_i)/\Tn$ gives a fiber bundle with fiber $\Tn$.  The orbit space $(X\backslash \cup_i X_i)/\Tn$ is simply connected because $X/\Tn$ ia a manifold with corners, $(X\backslash \cup_i X_i)/\Tn$ is the interior of $X/\Tn$ and $X/\Tn$ is simply connected by assumption.  Therefore the inclusion map from a free $\Tn$-orbit to $X\backslash \cup_i X_i$ induces an isomorphism on their fundamental groups.  But any free $\Tn$-orbit shrinks to a fixed point in $X$, so the epimorphism in \eqref{eq:5.1} must be trivial and hence $X$ is simply connected.   
\end{proof}
	
	Now, we are in a position to prove the following main theorem stated in the Introduction.
	\begin{theo}\label{theo:todd1}
		If a complex torus manifold $X$ has vanishing odd degree cohomology, then the Todd genus of $X$ is $1$.
	\end{theo}
	\begin{proof}
	Let $n$ be the complex dimension of $X$ as usual.  
Since $H^{\text{odd}}(X)=0$, the orbit space $X/\Tn$ is face-acyclic by Theorem~\ref{theo:Ma-Pa2}.  As remarked after Theorem~\ref{theo:Ma-Pa2}, one can change $X$ into a complex torus manifold whose orbit space is a homology polytope by blowing-up $X$ equivariantly.  Since Todd genus is a birational invariant, it remains unchanged under blowing-up. Therefore we may assume that the orbit space of our $X$ is a homology polytope, so that any intersection of characteristic submanifolds of $X$ is connected unless empty and $|\Sigma_X|$ is a homology sphere of dimension $n-1$.  
Since the orbit space of $X_i$ is a facet of $X/\Tn$, it is also a homology polytope so that any intersection of characteristic submanifolds of $X_i$ (viewed as a complex torus manifold) is also connected unless empty and $|\Sigma_{X_i}|$ is a homology sphere of dimension $n-2$.  Therefore, the results in Section~\ref{sec:map} are applicable to $X$ and $X_i$'s. 

We shall prove the theorem by induction on the complex dimension $n$ of $X$. If $n=1$, then $X$ is $\C P^1$ and hence $\Todd(X)=1$. When $n=2$, the orbit space $X/\T2$ is contractible because $X/\T2$ is acyclic by Theorem \ref{theo:Ma-Pa2} and the dimension of $X/\T2$ is $2$.  Therefore, $X$ is simply connected by Lemma~\ref{lemm:5.1} and $\Todd(X)=1$ by Proposition~\ref{prop:todd}.  

Assume that $n \geq 3$ and the theorem holds when the complex dimension is equal to $n-1$. Then, $\Todd(X_i)=1$ for any $X_i$ by induction assumption and hence $f_{X_i} : |\Sigma _{X_i}| \to S^{n-2}$ is a homeomorphism by Lemma~\ref{lemm:homeo}. Since $|\Sigma _X|$ is a homology sphere of dimension $n-1(\ge 2)$,  $|\Sigma_X|$ is connected and hence $f_X : |\Sigma _X | \to S^{n-1}$ is a homeomorphism by Corollary \ref{coro:covering}. It follows from Lemma \ref{lemm:homeo} that $\Todd(X)=1$. This completes the induction step and the theorem is proved.
	\end{proof}	

	\section{Proof of Theorem~\ref{theo:quasitoric}}\label{sec:quasitoric}
A quasitoric manifold $X$ of dimension $2n$ is a smooth closed manifold endowed with a locally standard $\Tn$-action, whose orbit space is a simple polytope $Q$ of dimension $n$. Clearly, $X$ is a torus manifold.  The characteristic submanifolds $X_1,\dots,X_m$ of $X$ bijectively correspond to the facets $Q_1,\dots, Q_m$ of $Q$ through the quotient map $\pi \colon X \to Q$.  
Therefore, for $I\subset \{1,\dots,m\}$, $X_I=\cap_{i\in I}X_i$ is non-empty if and only if $Q_I:=\cap_{i\in I}Q_i$ is non-empty; so  
the simplicial complex 
\[
\Sigma_X=\{ I\subset \{1,\dots,m\}\mid X_I\neq\emptyset\}
\]
introduced in Section~\ref{sec:multi-fan} is isomorphic to the boundary complex of the simplicial polytope dual to $Q$.   
As before, let $T_i$ be the circle subgroup of $\Tn$ which fixes $X_i$ pointwise and let $\lambda_i\colon S^1\to T_i\subset \Tn$ be an isomorphism.  There are two choices of $\lambda_i$ for each $i$.  

One can recover $X$ from the data $(Q,\{\lambda_i\}_{i=1}^m)$ up to equivariant homeomorphism as follows. 
Any codimension $k$ face $F$ of $Q$ is written as $Q_I$ for a unique $I\in \Sigma_X$ with cardinality $k$ and we denote the subgroup $T_I$ by $T_F$.  
For a point $p \in Q$, we denote by $F(p)$ the face containing $p$ in its relative interior. Set 
	\begin{equation*}
		X(Q,\{\lambda_i\}_{i=1}^m ):= \Tn\times Q/\sim,
	\end{equation*}
	where $(t,p)\sim (s,q)$ if and only if $p=q$ and $ts^{-1} \in T_{F(p)}$. Then $X$ and $X(Q,\{\lambda_i\}_{i=1}^m)$ are known to be equivariantly homeomorphic (\cite{Da-Ja91}).

	Suppose that our quasitoric manifold $X$ admits an invariant complex structure.  Then, the isomorphism $\lambda _i$ is unambiguously determined by requiring the identity \eqref{eq:3.1}, that is  
	\begin{equation*}
		\lambda _i(g)_*(\xi )=g\xi, \quad ^\forall g \in S^1 \text{ and } ^\forall \xi \in TX|_{X_i}/TX_i.
	\end{equation*}  
The simplicial complex $\Sigma_X$ and the elements $\lambda_i$'s are used to define the multi-fan of $X$.  But since the Todd genus of $X$ is one by Theorem~\ref{theo:todd1}, the multi-fan of $X$ is an ordinary complete non-singular fan by Corollary~\ref{coro:fan} and hence it is the fan of a toric manifold.  
Finally, we note that since $\Sigma_X$ is the boundary complex of the simplicial polytope dual to the simple polytope $Q$, it determines $Q$ as a manifold with corners up to homeomorphism.  This implies Theorem \ref{theo:quasitoric} because the equivariant homeomorphism type of $X$ is determined by $Q$ and the elements $\lambda_i$'s as remarked above. 

\bigskip
{\bf Acknowledgment.}  The authors thank Masaaki Ue and Yoshinori Namikawa for their helpful comments on the automorphism groups of compact complex surfaces.

\end{document}